\documentclass{amsart}

\usepackage[english]{babel}
\usepackage{amssymb}
\usepackage{hyperref}
\usepackage{tikz}
\usepackage{booktabs}
\usepackage{multirow}
\usepackage{mathtools}
\usetikzlibrary{patterns}

\newcommand{\arXiv}[1]{\href{http://arxiv.org/abs/#1}{arXiv:#1}}

\theoremstyle{plain}
\newtheorem{prop}{Proposition}[section]
\newtheorem{lem}[prop]{Lemma}
\theoremstyle{definition}
\newtheorem{rem}[prop]{Remark}
\newtheorem{exmp}[prop]{Example}
\newtheorem{defn}[prop]{Definition}
\newtheorem{notation}[prop]{Notation}

\DeclareMathOperator{\mult}{mult}
\DeclareMathOperator{\MAX}{max}

\newcommand{\bN}{\mathbb{N}}
\newcommand{\graph}{\mathcal{G}}
\newcommand{\abs}[1]{\left|#1\right|}
\newcommand{\ubar}[1]{\underline{#1}}

\newcommand{\psm}[1]{\left(\begin{smallmatrix}#1\end{smallmatrix}\right)}

\title{Generating stable modular graphs}
\author{Stefano Maggiolo}
\author{Nicola Pagani}
\date{\today}

\begin{document}

\begin{abstract}
  We present and prove the correctness of the program
  \texttt{boundary}, whose sources are available at
  \href{http://people.sissa.it/~maggiolo/boundary/}
  {\texttt{http://people.sissa.it/\~{}maggiolo/boundary/}}. Given two
  natural numbers $g$ and $n$ satisfying $2g+n-2>0$, the program
  generates all genus $g$ stable graphs with $n$ unordered marked
  points. Each such graph determines the topological type of a nodal
  stable curve of arithmetic genus $g$ with $n$ unordered marked
  points. Our motivation comes from the fact that the boundary of the
  moduli space of stable genus $g$, $n$-pointed curves can be
  stratified by taking loci of curves of a fixed topological type.
\end{abstract}

\maketitle
\setcounter{tocdepth}{1}
\tableofcontents

\section{Introduction}

Moduli spaces of smooth algebraic curves have been defined and then
compactified in algebraic geometry by Deligne and Mumford in their
seminal paper~\cite{deligne_mumford}. A conceptually important
extension of this notion in the case of pointed curves was introduced
by Knudsen~\cite{knudsen}.

The points in the boundary of the moduli spaces of pointed, nodal
curves with finite automorphism group. These curves are called
\emph{stable curves} (or pointed stable curves). The topology of one
such curve is encoded in a combinatorial object, called \emph{stable
  graph}.  The boundary of the moduli space admits a topological
stratification, made of loci whose points are curves with a fixed
topological type and a prescribed assignment of the marked points on
each irreducible component.

The combinatorics of the stable graphs have been investigated in
several papers in algebraic geometry, for many different purposes (see
for instance~\cite{getzler_kapranov, vanopstall_veliche_1,
  vanopstall_veliche_2, yang_3}). Our aim with this program is to
provide a useful and effective tool to generate all the stable graphs
of genus $g$ with $n$ unordered marked points up to isomorphism, for
low values of $g$ and $n$.

We construct an algorithm to generate all the stable graphs of genus
$g$ with $n$ unordered marked points. Our program uses then the software
\texttt{nauty}~\cite{mckay} to eliminate isomorphic graphs from the
list of graphs thus created. Since to check that two stable graphs are
isomorphic is computationally onerous, we try to generate a low number
of stable graphs, provided that we want at least one for every
isomorphism class. The algorithm generates recursively the vectors of
genera, number of marked points, number of loops, and the adjacency
matrix. While it fills these data, it checks the stability condition
and the condition on the total genus as early as possible, in order to
minimize the time spent on the branches of the recursion that do not
lead to stable graphs. Some analysis of the algorithm's performances
can be seen in Section~\ref{sec:performance}.

Programs for enumerative computations on
$\overline{\mathcal{M}}_{g,n}$ have been implemented in both Maple and
Macaulay2~(\cite{faber, yang_2, yang_1}). Our program can be used,
for example, to improve the results of~\cite[Section 5]{yang_3}, or
to prove combinatorial results on the moduli space of pointed stable
curves with low genus (cfr.~\cite{busonero_melo_stoppino}, for
example Corollary~5.3).

\section{Stable modular graphs}

From now on, we fix two natural numbers $G$ and $N$ such that $2
G-2+N>0$.  For every $K \in \bN^+$, we define $\ubar{K} = \{0, \dots,
K-1\}$ and $\Sigma_K$ to be the symmetric group on~the~set~$\ubar{K}$.

\begin{defn}
  \mbox{}
  \begin{itemize}
  \item An \emph{undirected multigraph\/} $\graph$ is a couple $(V,
    E)$ with $V$ a finite set of \emph{vertices\/} and $E$ a finite
    multiset of \emph{edges\/} with elements in $V \times V/\Sigma_2$.
  \item The multiplicity of the edge $(v, w)$ in $E$ is denoted by
    $\mult(v, w)$.
  \item The \emph{total multiplicity\/} of $\graph$, or its
    \emph{number of edges}, is $\abs{E}$: the cardinality of $E$ as a
    multiset.
  \item The \emph{degree\/} of a vertex $v$ is $\deg v \coloneqq 2
    \mult(v, v) + \sum_{w \neq v} \mult(v, w)$.
  \item A \emph{colored undirected multigraph\/} is a multigraph with
    some additional data attached to each vertex.
  \end{itemize}
\end{defn}

\begin{defn}\label{def:stable graph}
  A \emph{stable graph\/} of type $(G, N)$ is a colored undirected
  multigraph $\graph = (V, E)$, subject to the following conditions.
  \begin{enumerate}
  \item The color of a vertex $v$ is given by a pair of natural
    numbers $(g_v, n_v)$. The two numbers are called respectively the
    \emph{genus} and the \emph{number of marked points} of the vertex
    $v$.
  \item\label{it:condition connected} $\graph$ is connected.
  \item\label{it:condition genus} Its \emph{total genus}, defined as
    $\sum_{v \in V} g_v + \abs{E} - (\abs{V} - 1)$, equals $G$.
  \item Its \emph{total number of marked points}, defined as $\sum_{v
      \in V} n_v$, equals $N$.
  \item\label{it:condition stability} Stability condition: $\deg v +
    n_v \geq 3$ for every vertex $v$ with $g_v = 0$.
  \end{enumerate}
\end{defn}

\begin{notation}
  The number $\deg v + n_v$ is often called the \emph{number of half
    edges\/} associated to the vertex $v$. Condition \ref{it:condition
    stability} can be rephrased in: for every vertex $v$ of genus $0$,
  its number of half edges is at least $3$.
\end{notation}

Two stable graphs $\graph = (V, E, g, n)$ and $\graph^\prime =
(V^\prime, E^\prime, g^\prime, n^\prime)$ are \emph{isomorphic\/} if
there is a bijection $f\colon V \to V^\prime$ such that:
\begin{itemize}
\item $\mult(v, w) = \mult(f(v), f(w))$ for every $v, w \in V$;
\item $g_v = g^\prime_{f(v)}$ and $n_v = n^\prime_{f(v)}$ for every $v
  \in V$.
\end{itemize}
Our task is to generate one stable graph for each isomorphism class.

\begin{rem}
  Note that from the definition just given, we are working with an
  unordered set of marked points. The output of the program are the
  boundary strata of the moduli space of stable, genus $g$ curves with
  $n$ unordered points $\overline{\mathcal{M}}_{g,n}/ \Sigma_n$.
\end{rem}

\section{Description of the algorithm}\label{sec:description}

In this section we describe the general ideas of our algorithm. Let us
first introduce the notation we use in the program.

\begin{notation}\label{not:gnla}
  The set of vertices $V$ will always be $\ubar{K}$, so that vertices
  will be identified with natural numbers $i, j, \dots$. The
  multiplicity of the edge between $i$ and $j$ will be denoted by
  $a_{i,j}$: the symmetric matrix $a$ is called the \emph{adjacency
    matrix} of the stable graph. For convenience, we will denote $l_j
  \coloneqq a_{j,j}$: it is the vector whose elements are the number
  of loops at the vertex $j$. For simplicity, we will consider $g_j$,
  $n_j$, $l_j$, $a_{i,j}$ to be defined also for $i$ or $j$ outside
  $\ubar{K}$, in which case their value is always assumed to be $0$.
\end{notation}

\begin{rem}
  In the following, we assume $\abs{V} > 1$ in order not to deal with
  degenerate cases. There are trivially $G+1$ stable graphs of type
  $(G, N)$ with one vertex. Indeed, if there is exactly one vertex,
  the choice of the genus uniquely determines the number of loops on
  it after Definition~\ref{def:stable graph}.
\end{rem}

The program uses recursive functions to generate the data that
constitute a stable graph. In order, it generates the numbers $g_j$,
then the numbers $n_j$, $l_j$ (the diagonal part of the matrix $a$),
and finally, row by row, a symmetric matrix representing $a$.

When all the data have been generated, it tests that all the
conditions of Definition~\ref{def:stable graph} hold, in particular
that the graph is actually connected and satisfies the stability
conditions. Then it uses the software \texttt{nauty}~\cite{mckay} to
check if this graph is isomorphic to a previously generated graph. If
this is not the case, it adds the graph to the list of graphs of genus
$G$ with $N$ marked points.

A priori, for each entry of $g$, $n$, $l$, and $a$ the program tries
to fill that position with all the integers. This is of course not
possible, indeed it is important to observe here that each datum is
bounded. From below, a trivial bound is $0$, that is, no datum can be
negative. Instead, a simple upper bound can be given for each entry of
$g$ by the number $G$, and for each entry of $n$ by the number
$N$. For $l$ and $a$, upper bounds are obtained from $G$ using the
condition on the total genus (Condition~\ref{def:stable graph}).

These bounds are coarse: Section \ref{sec:ranges} will be devoted to
proving sharper bounds, from above and from below. Also, we will make
these bounds dynamical: for instance assigning the value $g_0 > 0$
clearly lowers the bound for $g_j, j > 0$. The improvement of these
bounds is crucial for the performance of the algorithm. In any case,
once we know that there are bounds, we are sure that the recursion~terminates.

The algorithm follows this principle: we want to generate the smallest
possible number of couples of isomorphic stable graphs. To do so, we
generalize the idea that to generate a vector for every class of
vectors of length $K$ modulo permutations, the simplest way is to
generate vectors whose entries are increasing. The program fills the
data row by row in the matrix:
\begin{equation}\label{eq:big matrix}
  \begin{pmatrix}
    g_0 & g_1 & \cdots & g_{K-1}\\
    n_0 & n_1 & \cdots & n_{K-1}\\
    l_0 & l_1 & \cdots & l_{K-1}\\
    \hline
    \bullet & a_{0,1} & \cdots & a_{0,K-1}\\
    a_{1,0} & \bullet & \ddots & \vdots\\
    \vdots & \ddots & \bullet & a_{K-2,K-1}\\
    a_{K-1,0} & \cdots & a_{K-1,K-2} & \bullet
  \end{pmatrix}\text{,}
\end{equation}
and generates only matrices whose columns are ordered. Loosely
speaking, we mean that we are ordering the columns lexicographically,
but this requires a bit of care, for two reasons:
\begin{itemize}
\item the matrix $a$ needs to be symmetric; in the program we generate
  only the strictly upper triangular part;
\item the diagonal of $a$ need not be considered when deciding if a
  column is greater than or equal to the previous one.
\end{itemize}

Therefore, to be precise, we define a relation (\emph{order}) for
adjacent columns. Let us call~$c_{j-1}$ and $c_j$ two adjacent columns
of the matrix~\eqref{eq:big matrix}. They are said to be equivalent
if~$c_{j-1,i} = c_{j,i}$ for any $i \notin \{j-1+3, j+3\}$. If they
are not equivalent, denote with $i_0$ the minimum index such that $i_0
\notin \{ j-1+3, j+3\}$ and $c_{j-1,i_0} \neq c_{j,i_0}$. Then we
state the relation $c_{j-1} < c_j$ if and only if $c_{j-1,i_0} <
c_{j,i_0}$. We do not define the relation for non-adjacent columns.
We say that the data are ordered when the columns are weakly
increasing, that is if, for all $j$, either $c_{j-1}$ is equivalent to
$c_j$ or $c_{j-1} < c_j$.

To ensure that the columns are ordered (in the sense we explained
before), the program keeps track of \emph{divisions}. We start filling
the genus vector $g$ in a non decreasing way, and every time a value
$g_j$ strictly greater than $g_{j-1}$ is assigned, we put a division
before $j$. This means that, when assigning the value of $n_j$, we
allow the algorithm to start again from $0$ instead of $n_{j-1}$,
because the column $c_j$ is already bigger than the column $c_{j-1}$.

After completing $g$, we start filling the vector $n$ in such a way
that, within two divisions, it is non decreasing. Again we introduce a
division before $j$ every time we assign a value~$n_j$ strictly
greater than $n_{j-1}$. We follow this procedure also for the vector
$l$.

Finally, we start filling the rows of the matrix $a$. Here the
procedure is a bit different. Indeed even if for the purpose of
filling the matrix it is enough to deal only with the upper triangular
part, imposing the conditions that the columns are ordered involves
also the lower triangular part. A small computation gives that the
value of $a_{i,j}$ is assigned starting from:
\[
\begin{cases}
  0 & \text{if there are divisions before $i$ and $j$}\\
  a_{i,j-1} & \text{if there is a division before $i$ but not before
    $j$}\\
  a_{i-1,j} & \text{if there is a division before $j$ but not before
    $i$}\\
  \max\{a_{i,j-1}, a_{i-1,j}\} & \text{if there are no divisions
    before $i$ or $j$,}
\end{cases}
\]
and we put a division before $i$ if $a_{i,j} > a_{i-1,j}$ and a
division before $j$ if $a_{i,j} > a_{i,j-1}$.

We cannot conclude immediately that this procedure gives us all
possible data up to permutations as in the case of a single
vector. This is because the transformation that the whole matrix
undergoes when a permutation is applied is more complicated: for the
first three rows (the vectors $g$, $n$, $l$), it just permutes the columns,
but for the remaining rows, it permutes both rows and columns. Indeed,
to prove that the procedure of generating only ordered columns does
not miss any stable graph is the content of the following section.

\section{The program generates all graphs}\label{sec:proof}

We want to prove the following result.

\begin{prop}\label{prop:main}
  The algorithm described in the previous section generates at least
  one graph for every isomorphism class of stable graphs.
\end{prop}

From now on, besides $G$ and $N$, we also fix the number of vertices
$K$, and focus on proving that the algorithm generates at least one
graph for every isomorphism class of stable graphs with $K$ vertices.

\begin{notation}
  We have decided previously to encode the data of a stable graph in a
  $(K+3 \times K)$ matrix $G \coloneqq (g, n, l, a)$
  (cfr.~\eqref{eq:big matrix}). We denote by $\mathcal{A}$ the set of
  all such matrices, and by $\mathcal{M}$ the set of all $(K+3 \times
  K)$ matrices that are generated by the algorithm described in the
  previous section.
\end{notation}

We can assume that the graphs generated by the algorithm are stable,
since we explicitly check connectedness and stability. In other words,
we can assume the inclusion $\mathcal{M} \subset \mathcal{A}$. Hence,
in order to prove Proposition~\ref{prop:main}, we will show that every
$G \in \mathcal{A}$ is in $\mathcal{M}$ up to applying a permutation
of $\ubar{K}$. The idea is to give a characterization
(Lemma~\ref{lemma:char}) of the property of being an element of
$\mathcal{M}$.

Recall first that the algorithm generates only matrices whose columns
are ordered, as described in Section~\ref{sec:description}. More
explicitly, if $G = (g, n, l, a) \in \mathcal{A}$, then $G \in
\mathcal{M}$ if and only if:
\begin{multline*}
  \forall (i,j)\colon
  i \not\in \{j-1, j\},\\
  \begin{aligned}
    g_{j-1} &> g_j &&\text{does not happen,}\\
    n_{j-1} &> n_j &\Rightarrow\  & g_{j-1} < g_j\,\text{,}\\
    l_{j-1} &> l_j &\Rightarrow\  & g_{j-1} < g_j \vee n_{j-1} < n_j\,\text{, and}\\
    a_{i,j-1} &> a_{i,j} &\Rightarrow\ & g_{j-1} < g_j \vee n_{j-1} < n_j \vee l_{j-1} < l_j \vee\\
    &&&\ \exists i^\prime < i: i^\prime \not\in \{j-1,j\} \wedge a_{i^\prime,j-1} < a_{i^\prime,j}\,\text{.}
  \end{aligned}
\end{multline*}

Let us call a piece of data $g_j$, $n_j$, $l_j$, or $a_{i,j}$ a
\emph{breaking position\/} if it does not satisfy the condition
above. Observe that a matrix $G \in \mathcal{A}$ has a breaking
position if and only if $G$ is not an element of $\mathcal{M}$.

We now introduce a total order on the set $\mathcal{A}$ of matrices $G
= (g,n,l,a)$. If $G$ is such a matrix, let $v(G)$ be the vector
obtained by juxtaposing the vectors $g$, $n$, $l$ and the rows of the
upper triangular part of $a$. For example, if
\[
  G = \begin{pmatrix}
    0 & 0 & 2 & 0\\
    1 & 1 & 0 & 1\\
    0 & 0 & 0 & 0\\
    \hline
    \bullet & 1 & 1 & 1\\
    1 & \bullet & 2 & 1\\
    1 & 2 & \bullet & 0\\
    1 & 1 & 0 & \bullet
  \end{pmatrix}
\]
(with the same structure as~\eqref{eq:big matrix}), then we define
\[
v(G) \coloneqq (0, 0, 2, 0,\quad 1, 1, 0, 1,\quad 0, 0, 0, 0,\quad 1, 1,
1,\quad 2, 1,\quad 0)\,\text{.}
\]

\begin{defn}\label{def:order}
  If $G, H \in \mathcal{A}$, we write $G \prec H$ if and only if
  $v(G)$ is smaller than $v(H)$ in the lexicographic order.  In this
  case we say that the matrix $G$ is smaller than the matrix $H$.
\end{defn}

Note that this total order on the set of matrices must not be confused
with the partial order described in
Section~\ref{sec:description}. From now on we will always refer to the
latter order on $\mathcal{A}$.

\begin{rem}
  If $\sigma \in \Sigma_K$ is a permutation and $G = (g, n, l, a)$ is
  a graph, then we can apply $\sigma$ to the entries of the data of
  $G$, obtaining an isomorphic graph. The action of $\sigma$ on $G$
  is: $(g,n,l,a)\to (g^\prime,n^\prime,l^\prime,a^\prime)$ where
  $g^\prime_j=g_{\sigma(j)}$, $n^\prime_j=n_{\sigma(j)}$,
  $l^\prime_j=l_{\sigma(j)}$ and $a^\prime_{i,j} =
  a_{\sigma(i),\sigma(j)}$. We denote this new matrix by $\sigma
  G$. We write $\sigma_{i,j}$ for the element of $\Sigma_K$ that
  corresponds to the transposition of $i, j \in \ubar{K}$.
\end{rem}

Now we are able to state the characterization we need to prove
Proposition~\ref{prop:main}.

\begin{lem}\label{lemma:char}
  Let $G \in \mathcal{A}$; then $G \in \mathcal{M}$ if and only if $G$
  is minimal in the set
  \[
  \bigl\{ \sigma_{j-1,j} G \,\mid\, 0<j<K \bigr\}\text{.}
  \]
  with respect to the order given in Definition~\ref{def:order}.
\end{lem}

\begin{proof}
  We will prove that $G$ is not minimal if and only if there is a
  breaking position.

  Assume there is at least one breaking position in $G$. If there is
  one in $g$, $n$, or $l$, it is trivial to see that transposing the
  corresponding index with the previous one gives a smaller matrix. If
  this is not the case, let $a_{i,j}$ be a breaking position such that
  $a_{i^\prime,j}$ is not a breaking position whenever $i^\prime < i$
  (the position $(i,j)$ is the first breaking position of its
  column). We deduce that $g_{j-1} = g_j$, $n_{j-1} = n_j$, $l_{j-1} =
  l_j$, and that for all $i^\prime < i$ not in $\{j-1,j\}$, we have
  $a_{i^\prime, j-1} = a_{i^\prime, j}$. Let $H \coloneqq
  \sigma_{j-1,j} G$; the vectors $g$, $n$, and $l$ (the first three
  rows) coincide in $G$ and $H$.
  \begin{itemize}
  \item If $j > i$, the smallest breaking position is in the upper
    triangular part of $a$; it is then clear that $H \prec G$.
  \item If $j < i$, the smallest breaking position is in the lower
    triangular part; by using the symmetry of the matrix $a$ we again
    obtain $H \prec G$ (see the right part of Figure~\ref{fig:a}).
  \end{itemize}

  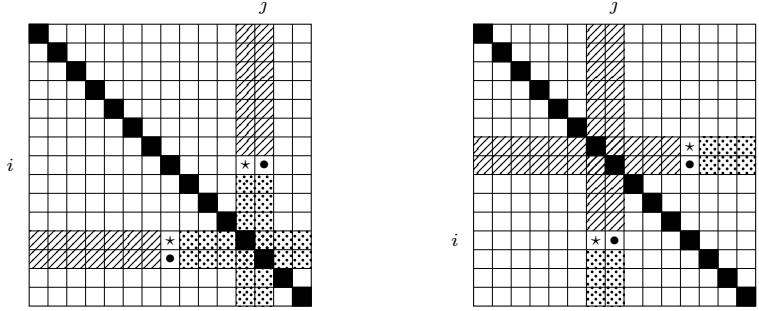
\begin{figure}[t]
    \centering
    \begin{tikzpicture}[xscale=0.25,yscale=-0.25]
      \node at (-1,7.5) {$\scriptstyle{i}$};
      \node at (12.5,-1) {$\scriptstyle{j}$};
      \node at (11.5,7.5) {$\scriptstyle{\star}$};
      \node at (12.5,7.5) {$\scriptstyle{\bullet}$};
      \node at (7.5,11.5) {$\scriptstyle{\star}$};
      \node at (7.5,12.5) {$\scriptstyle{\bullet}$};
      \fill [pattern=crosshatch dots] (8,11) -- (8,13) -- (15,13) -- (15,11) -- cycle;
      \fill [pattern=crosshatch dots] (11,8) -- (13,8) -- (13,15) -- (11,15) -- cycle;
      \fill [pattern=north east lines] (0,11) -- (0,13) -- (7,13) -- (7,11) -- cycle;
      \fill [pattern=north east lines] (11,0) -- (13,0) -- (13,7) -- (11,7) -- cycle;

      \draw (0,0) -- (15,0) -- (15,15) -- (0,15) -- cycle;
      \foreach \x in {1,2,...,14}
      {
        \draw [very thin] (0,\x) -- (15,\x);
        \draw [very thin] (\x,0) -- (\x,15);
      }
      \foreach \x in {0,1,...,14}
      \filldraw (\x,\x) -- (\x,\x+1) -- (\x+1,\x+1) -- (\x+1,\x) -- cycle;
    \end{tikzpicture}
    \hspace{1.5cm}
    \begin{tikzpicture}[xscale=0.25,yscale=-0.25]
      \node at (-1,11.5) {$\scriptstyle{i}$};
      \node at (7.5,-1) {$\scriptstyle{j}$};
      \node at (6.5,11.5) {$\scriptstyle{\star}$};
      \node at (7.5,11.5) {$\scriptstyle{\bullet}$};
      \node at (11.5,6.5) {$\scriptstyle{\star}$};
      \node at (11.5,7.5) {$\scriptstyle{\bullet}$};
      \fill [pattern=crosshatch dots] (12,6) -- (12,8) -- (15,8) -- (15,6) -- cycle;
      \fill [pattern=crosshatch dots] (6,12) -- (8,12) -- (8,15) -- (6,15) -- cycle;
      \fill [pattern=north east lines] (0,6) -- (0,8) -- (11,8) -- (11,6) -- cycle;
      \fill [pattern=north east lines] (6,0) -- (8,0) -- (8,11) -- (6,11) -- cycle;

      \draw (0,0) -- (15,0) -- (15,15) -- (0,15) -- cycle;
      \foreach \x in {1,2,...,14}
      {
        \draw [very thin] (0,\x) -- (15,\x);
        \draw [very thin] (\x,0) -- (\x,15);
      }
      \foreach \x in {0,1,...,14}
      \filldraw (\x,\x) -- (\x,\x+1) -- (\x+1,\x+1) -- (\x+1,\x) -- cycle;
    \end{tikzpicture}\vspace{-.5\baselineskip}
    \caption{The matrix $a$ when the first breaking position (the
      bullet) is $a_{i,j}$ with $j > i$ (left) or $j < i$
      (right). When transposing $j-1$ and $j$, the white and the
      diagonal-filled entries do not change.}\label{fig:a}
  \end{figure}

  Conversely, let $j$ be such that $H \coloneqq \sigma_{j-1,j} G \prec
  G$. Then consider the first entry (reading from left to right) of
  the vector $v(G)$ that is strictly bigger than $v(H)$. This is a
  breaking position. Notice that if it occurs in the matrix $a$
  (equivalently, in the last $K$ rows), it is actually the first
  breaking position of its column.
\end{proof}

The proof of Proposition~\ref{prop:main} follows arguing as in this
example.

\begin{exmp}
  Let $G_0 \coloneqq G \in \mathcal{A}$ be the graph of the previous
  example:
  \[
  G_0 = \psm{
    0 & 0 & 2 & 0\\
    1 & 1 & 0 & 1\\
    0 & 0 & 0 & 0\\[2pt]
    \hline\\[1pt]
    \bullet & 1 & 1 & 1\\
    1 &\bullet & 2 & 1\\
    1 & 2 & \bullet & 0\\
    1 & 1 & 0 & \bullet
  }\text{.}
  \]
  This graph is stable but not in $\mathcal{M}$ because, for example,
  $g_2 > g_3$ implies that $g_3$ is a breaking position. Thus we apply
  the permutation $\sigma_{2,3}$, obtaining the graph
  \[
  G_1 \coloneqq \sigma_{2,3} G_0 = \psm{
    0 & 0 & 0 & 2\\
    1 & 1 & 1 & 0\\
    0 & 0 & 0 & 0\\[2pt]
    \hline\\[1pt]
    \bullet & 1 & 1 & 1\\
    1 & \bullet & 1 & 2\\
    1 & 1 & \bullet & 0\\
    1 & 2 & 0 & \bullet
  } \prec G_0\text{.}
  \]
  Now $a_{3,2}$ is a breaking position; applying $\sigma_{1,2}$, we
  obtain
  \[
  G_2 \coloneqq \sigma_{1,2} G_1 = \psm{
    0 & 0 & 0 & 2\\
    1 & 1 & 1 & 0\\
    0 & 0 & 0 & 0\\[2pt]
    \hline\\[1pt]
    \bullet & 1 & 1 & 1\\
    1 & \bullet & 1 & 0\\
    1 & 1 & \bullet & 2\\
    1 & 0 & 2 & \bullet
  } \prec G_1\text{.}
  \]
  This introduces a new breaking position at $a_{3,1}$, so we apply
  the transposition $\sigma_{0,1}$:
  \[
  G_3 \coloneqq \sigma_{0,1} G_2 = \psm{
    0 & 0 & 0 & 2\\
    1 & 1 & 1 & 0\\
    0 & 0 & 0 & 0\\[2pt]
    \hline\\[1pt]
    \bullet & 1 & 1 & 0\\
    1 & \bullet & 1 & 1\\
    1 & 1 & \bullet & 2\\
    0 & 1 & 2 & \bullet
  } \prec G_2\text{.}
  \]
  The graph $G_3$ is finally in $\mathcal{M}$ and indeed no
  transposition can make it smaller.
\end{exmp}

\begin{proof}[Proof of Proposition~\ref{prop:main}]
  Recall that we have to prove that for every $G \in \mathcal{A}$,
  there is a permutation $\sigma \in \Sigma_K$ such that $\sigma G \in
  \mathcal{M}$.

  So, let $G_0 = G \in \mathcal{A}$. If $G \in \mathcal{M}$, then we
  are done; otherwise, $G$ does not satisfy the condition of
  Lemma~\ref{lemma:char}, hence there is a transposition $\sigma_{j-1,
    j}$ such that $G_1 = \sigma_{j-1,j} G_0 \prec G_0$.

  The iteration of this process comes to an end (that is, we arrive to
  a matrix in $\mathcal{M}$) since the set
  \[
  \bigl\{ \sigma G \mid \sigma \in \Sigma_K\bigr\}
  \]
  is finite.
\end{proof}

\section{Description of the ranges}\label{sec:ranges}

In Section~\ref{sec:description} we have introduced the algorithm, by
describing the divisions. In this section we introduce accurate ranges
for the possible values of $g$, $n$, $l$ and $a$.

We will deduce from the conditions of Definition~\ref{def:stable
  graph} some other necessary conditions that can be checked before
the graph is defined in its entirety. More precisely, every single
datum is assigned trying all the possibilities within a range that
depends upon the values of $G$ and $N$, and upon the values of the
data that have already been filled. The conditions we describe in the
following are not the only ones possible; we tried other possibilities,
but heuristically the others we tried did not give any improvement.

The order in which we assign the value of the data is $g$, $n$, $l$, and
finally the upper triangular part of $a$ row after row.

\begin{notation}
  Suppose we are assigning the $i$-th value of one of the vectors
  $g$, $n$ or $l$, or the $(i,j)$-th value of $a$. We define the
  following derived variables $e^{\MAX}$, $c$ and $p_1$ that depend
  upon the values that have already been assigned to $g$, $n$, $l$, $a$.

  We let $e^{\MAX}$ be the maximum number of edges that could be
  introduced in the subsequent iterations of the recursion, and $c$ be
  the number of couples of (different) vertices already connected by
  an edge. We let $p_1$ be the number of vertices $z$ to which the
  algorithm has assigned $g_z = 0$.  Note that the final value of
  $p_1$ is determined when the first genus greater than $0$ is
  assigned, in particular the final value of $p_1$ is determined at
  the end of the assignment of the values to the vector $g$.  On the
  other hand, $c$ starts to change its value only when the matrix $a$
  begins to be filled.

  \emph{After} the assignment of the $i$-th value, the derived values
  $e^{\MAX}$, $c$ and $p_1$ are then updated according to the
  assignment itself.
\end{notation}

\begin{notation} \label{not:partial_assign}
  When deciding $g$, $n$, or $l$, we let $n^{(2)}_i$ be the minimum
  between $2$ and the number of half edges already assigned to the
  $i$-th vertex. This is justified by the fact that we know that, when
  we will fill the matrix $a$, we will increase by one the number of
  half edges at the vertex $i$ in order to connect it to the rest of
  the graph. Hence, whenever $g_i = 0$, $n^{(2)}_i$ is the number of
  \emph{stabilizing\/} half edges at the vertex $i$: one half edge is
  needed to connect the vertex to the rest of the graph, and then at
  least two more half edges are needed to stabilize the vertex. When
  deciding $a_{i,j}$, it is also useful to have defined $h_i$, the
  total number of half edges that hit the $i$-th vertex. Finally, we
  define
  \begin{align*}
    G_i &\coloneqq \sum_{i^\prime < i} g_{i^\prime}\,\text{,} &
    N_i &\coloneqq \sum_{\substack{i^\prime < i}} n_{i^\prime}\,\text{,}\\
    N^{(2)} &\coloneqq \sum_{g_{i^\prime} = 0} n^{(2)}_{i^\prime}\,\text{,} &
    N^{(2)}_i &\coloneqq \sum_{\substack{i^\prime < i\\g_{i^\prime} = 0}} n^{(2)}_{i^\prime}\,\text{;}\\
    L_i &\coloneqq \sum_{i^\prime < i} l_{i^\prime}\,\text{,} &
    A_{i,j} &\coloneqq \sum_{\mathclap{i^\prime < i \vee j^\prime < j}} a_{i^\prime, j^\prime}\,\text{.}
  \end{align*}
\end{notation}

We are now ready to describe the ranges in which the data can vary.
We study subsequently the cases of $g$, $n$, $l$ and $a$, thus
following the order of the recursions of our algorithm. Each range is
described by presenting a first list of general constraints on the
parameters and then by presenting a second list containing the actual
ranges in the last line.

\subsection{Range for $g_i$}

When the algorithm is deciding the value of $g_i$, we have the
following situation:
\begin{itemize}
\item $e^{\MAX} = G - G_i + K - 1$ by Condition~\ref{it:condition
    genus};
\item amongst the $e^{\MAX}$ edges, there are necessarily $K-1$
  non-loop edges (to connect the graph); these $K-1$ edges give one
  half edge for each vertex, whereas we can choose arbitrarily where
  to send the other $K-2$ half edges; conversely, the $2(e^{\MAX} - K
  +1)$ half edges of the remaining edges can be associated to any
  vertex; therefore, the maximum number of half edges (not counting
  those that are needed to connect the graph) is $2e^{\MAX} - K + N =
  2(G - G_i) + K - 2 + N$;
\item we need $2p_1$ half edges to stabilize the genus $0$ vertices,
  since one half edge comes for free from the connection of the graph.
\end{itemize}

We use the following conditions to limit the choices we have for
$g_i$:
\begin{enumerate}
\item since $g$ is the first vector to be generated, there is no
  division before $i$, hence
  \[
  g_i \geq g_{i-1}\text{;}
  \]
  remember that $g_j = 0$ whenever $j \not\in \ubar{K}$;
\item we need at least $K-1$ non-loop edges, hence (using the fact
  that $\sum_{j \geq i} g_j \geq (K-i) g_i$)
  \begin{align*}
    &e^{\MAX} \geq K-1\\
    &\quad\Rightarrow G - G_i - (K-i) g_i + K-1 \geq K-1\\
    &\quad\Rightarrow (K-i)g_i \leq G - G_i\,\text{;}
  \end{align*}
\item in order to stabilize the $p_1$ vertices of genus $0$ (using the
  fact that one stabilizing half edge comes for free by connection) we
  must have
  \begin{align*}
    &2 p_1 \leq 2e^{\MAX} - K + N\\
    &\quad\Rightarrow 2p_1 \leq G - G_i - (K-i)g_i - K + N\\
    &\quad\Rightarrow (K-i)g_i \leq G - G_i - K + N - 2p_1\,\text{.}
  \end{align*}
\end{enumerate}

\subsection{Range for $n_i$}

When deciding $n_i$, we have the following situation:
\begin{itemize}
\item as before, $e^{\MAX} = G - G_K + K - 1 \geq K-1$, and the
  maximum number of half edges still to be assigned is $2e^{\MAX} - K
  + N - N_i - n_i = 2(G - G_K) + K - 2 + N - N_i - n_i$;
\item we need $2p_1 - N^{(2)}_i - n^{(2)}_i$ half edges to stabilize
  the first $p_1$ vertices;
\item if $g_i = 0$, we need $2(i+1) - N^{(2)}_i - n^{(2)}_i$ more half
  edges to stabilize the first $i+1$ vertices.
\end{itemize}

The following conditions define then the ranges for the possible
choices for $n_i$:

\begin{enumerate}
\item if there is not a division before $i$ (that is, if $g_i =
  g_{i-1}$), then we require $n_i \geq n_{i-1}$; otherwise, just $n_i
  \geq 0$;
\item we cannot assign more than $N$ marked points, hence (where we
  treat the case of $g_i = 0$ in a special way)
  \begin{align*}
    &N_i + n_i \leq N\\
    &\quad\Rightarrow n_i \leq N - N_i\\
    &\quad\Rightarrow (p_1 - i)n_i \leq N - N_i\text{ if moreover $g_i = 0$.}
  \end{align*}
\item if $g_i = 0$, for the purpose of stabilizing the first $i+1$
  curves we cannot use marked points anymore, therefore we have
  \begin{align*}
    &2 (i+1) - N^{(2)}_i - n^{(2)}_i \leq \bigl(2(G - G_K) + K - 2\bigr)\\
    &\quad\Rightarrow n^{(2)}_i = \min(2, n_i) \geq - \bigl(2(G - G_K) + K - 2) + (2(i+1) - N^{(2)}_i\bigr)\\
    &\quad\Rightarrow
    \begin{cases}
      \text{impossible} & \text{if $\mathrm{RHS} > 2$}\\
      n_i \geq \mathrm{RHS} & \text{otherwise.}
    \end{cases}
  \end{align*}
\end{enumerate}

\subsection{Range for $l_i$}

\enlargethispage{1.1\baselineskip}
When deciding $l_{i}$, this is the situation:
\begin{itemize}
\item $e^{\MAX} = G - G_K - L_i - l_i + K - 1 \geq K-1$, and the
  maximum number of half edges still to assign is $2e^{\MAX} - K = 2(G
  - G_K - L_i - l_i) + K - 2$;
\end{itemize}

The conditions on $l_i$ are then the following:
\begin{enumerate}
\item if there is not a division before $i$, then we require $l_i \geq
  l_{i-1}$; otherwise, just $l_i \geq 0$;
\item we need at least $K-1$ non-loop edges, hence
  \begin{align*}
    &e^{\MAX} \geq K-1\\
    &\quad\Rightarrow G - G_K - L_i - l_i + K-1 \geq K-1\\
    &\quad\Rightarrow l_i \leq G - G_K - L_i\,\text{;}
  \end{align*}
\item let $z$ be the index of the genus $0$ vertex with the least
  number of stabilizing half edges such that $z < i$; it already has
  $n_z + 2l_z$ half edges, but we cannot use loops anymore to
  stabilize it; hence,
  \begin{align*}
    &\max(0, 2-n_z-2l_z) \leq G - G_K - L_i - l_i + K - 1\\
    &\quad\Rightarrow l_i \leq G - G_K - L_i + K - 3 + n_z + 2l_z
  \end{align*}
\item assume $g_i = 0$; if $l_i > 0$, we are adding to the $i$-th
  vertex $2-n^{(2)}_i$ stabilizing half edges, and to stabilize the
  $p_1$ genus $0$ vertices, we need to have
  \begin{align*}
    &2 p_1 - N^{(2)} - \bigl(2-n^{(2)}_i\bigr) \leq 2e^{\MAX} - K\\
    &\quad\Rightarrow 2p_1 - N^{(2)} - \bigl(2-n^{(2)}_i\bigr) \max(0, 2-m_i) \leq 2(G - G_K - L_i - l_i + K - 1) - K\\
    &\quad\Rightarrow 2l_i \leq 2(G - G_K - L_i) + K + N^{(2)} - n^{(2)}_i - 2p_i\,\text{.}
  \end{align*}
\item assume $g_i = 0$; after deciding $l_i$, we still have $e^{\MAX}$
  edges to place, and each of them can contribute with one half edge
  to the stabilization of the $i$-th vertex; moreover, one of these
  half edges is already counted for the stabilization; hence
  \begin{align*}
    &n_i + 2l_i + (e^{\MAX} - 1) \geq 2\\
    &\quad\Rightarrow n_i + 2l_i + G - G_K - L_i - l_i + K - 1 - 1 \geq 2\\
    &\quad\Rightarrow l_i \geq 4 - n_i - G + G_K + L_i - K\,\text{.}
  \end{align*}
\end{enumerate}

\subsection{Range for $a_{i,j}$}

When deciding $a_{i,j}$, this is the situation:
\begin{itemize}
\item earlier in Notation \ref{not:partial_assign}, we observed that
  for the purpose of filling the vectors $g$,$n$ and $l$ we could
  consider a genus $0$ vertex stabilized when it had at least two
  half-edges (since the graph is going to be connected
  eventually). When assigning the values of $a$, the stability
  condition goes back to its original meaning, i.e. each vertex has at
  least $3$ half edges.
\item $e^{\MAX} = G - G_K - L_K - A_{i,j} + K - 1$;
\item we have already placed edges between $c$ couples of different
  vertices;
\end{itemize}

Here are the constraints that $a_{i,j}$ must satisfy:
\begin{enumerate}
\item if there is not a division before $i$, then we require $a_{i,j}
  \geq a_{i-1,j}$; otherwise, just $a_{i,j} \geq 0$;
\item if there is not a division before $j$, then we require $a_{i,j}
  \geq a_{i,j-1}$;
\item we need at least $K-2-c$ (if positive) edges to connect the
  graph, because if $a_{i,j} > 0$, $c$ will increase by $1$ (this
  estimate could be very poor, but enforcing the connectedness
  condition in its entirety before completing the graph is too slow),
  hence:
  \begin{align*}
    &e^{\MAX} - a_{i,j} \geq \max(0, K-2-c)\\
    &\quad\Rightarrow a_{i,j} \leq G - G_K - L_K - A_{i,j} +K - 1 - \max(0, K-2-c)\,\text{;}
  \end{align*}
\item $a_{i,j}$ contributes with at most $\max(0, 3-h_i) + \max(0,
  3-h_j)$ stabilizing half edges; hence, to stabilize the $p_1$ genus
  $0$ vertices, we need
  \begin{align*}
    &3p_1 - \sum_{g_{i^\prime} = 0} \min(3, n_i + 2l_i) - \bigl(\max(0, 3-h_i) + \max(0, 3-h_j)\bigr) \leq 2 (e^{\MAX} - a_{i,j})\\
    &\quad\Rightarrow 3p_1 - \sum_{g_{i^\prime} = 0} \min(3, n_i + 2l_i) - \bigl(\max(0, 3-h_i) + \max(0, 3-h_j)\bigr) \leq \\
    &\quad\quad\quad \leq 2 (G - G_K - L_K - A_{i,j} + K - 1 - a_{i,j})\\
    &\quad\Rightarrow 2a_{i,j} \leq 2 (G - G_K - L_K - A_{i,j} + K - 1) - 3p_1 +\\
    &\quad\quad\quad +\sum_{g_{i^\prime} = 0} \min(3, n_i + 2l_i) + \max(0, 3-h_i) + \max(0, 3-h_j)\,\text{.}
  \end{align*}
\item if $j = K-1$ (that is, if this is the last chance to add half
  edges to the $i$-th vertex), then we add enough edges from $i$ to
  $K-1$ in order to stabilize the vertex $i$; moreover, if up to now
  we did not place any non-loop edge on the vertex $i$, we impose
  $a_{i,K-1} > 0$.
  \begin{align*}
    a_{i,K-1}&>0  &&\text{ if $ a_{i,j} = 0$ for all $1<j<K-1$,}\\
    a_{i,K-1}&\geq 3-h_i &&\text{ if $g_i = 0$.}
  \end{align*}
\end{enumerate}

\section{Performance}\label{sec:performance}

The complexity of the problem we are trying to solve is intrinsically
higher than polynomial, because already the amount of data to generate
increases (at least) exponentially with the genera and the number of
marked points. We also observed an exponential growth of the ratio
between the time required to solve an instance of the problem and the
number of graphs generated. Anyway, our program is specifically
designed to attack the problem of stable graphs, and it can be
expected to perform better than any general method to generate graphs
applied to our situation.

We present here some of the results obtained when testing our program on
an Intel\textregistered{} Core\texttrademark 2 Quad Processor Q9450 at
2.66\,GHz. The version we tested is not designed for parallel
processing, hence it used only one of the four cores available.

However, when computing a specific graph, the program needs to keep in
the memory only the graphs with the same values in the vectors $g$,
$n$, $l$: memory usage becomes therefore negligible. Moreover this shows that
we can assign the computations of stable graphs with prescribed $g$,
$n$, $l$ to different cores or \textsc{cpu}s, thus having a highly
parallelized implementation of the program.

\begin{figure}[t]
  \centering
  \includegraphics{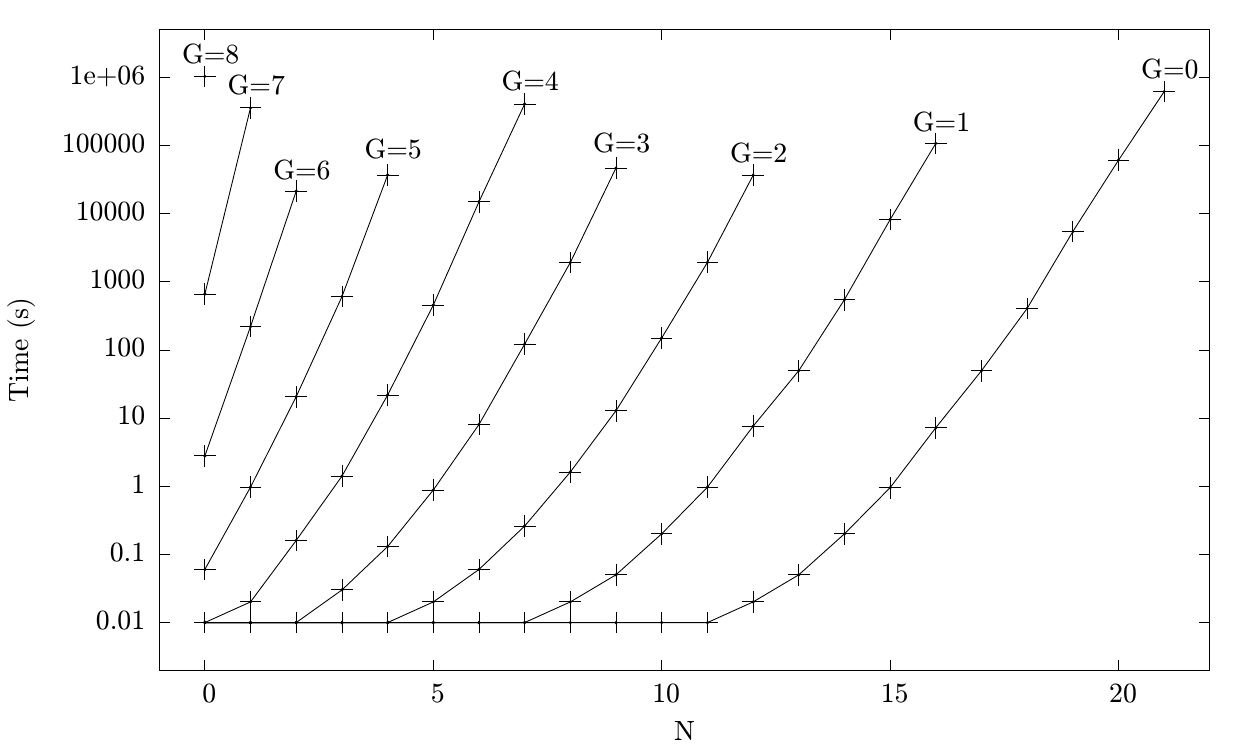}\vspace{-0.5\baselineskip}
  \caption{Time needed to compute all stable graphs of type $(G,
    N)$.}\label{fig:plot}
\end{figure}

\begin{table}[t]
  \centering
  \begin{tabular}{rrrrr}
    $G$ & $N$ & Time (s) & \# stable graphs & Duplicates (\%)\\
    \hline
    0 & 18 & 392 &   847,511 & 54.9\\
    1 & 14 & 539 & 1,832,119 & 41.3\\
    2 & 10 & 147 & 1,282,008 & 30.9\\
    3 &  7 & 117 & 1,280,752 & 29.9\\
    4 &  5 & 459 & 2,543,211 & 40.1\\
    5 &  3 & 606 & 2,575,193 & 54.7\\
    6 &  1 & 226 &   962,172 & 70.6\\
    7 &  0 & 681 & 1,281,678 & 85.6
  \end{tabular}\vspace{\baselineskip}
  \caption{For small $G$, the maximum $N$ such that all stable graphs
    of type $(G,N)$ can be computed in less than $15$ minutes.
   In the last column we show the ratio of duplicated graphs among the total number of those created by our generation algorithm.}
  \label{tab:number}
\end{table}

In Table~\ref{tab:number} we list, for each genus $G$, the maximum
number of marked points $N$ for which we can compute all the stable
graphs of type $(G, N)$ under 15 minutes.

In Figure~\ref{fig:plot} we show all the couples $(G, N)$ that we
computed against the time needed; the lines connect the results
referring to the same genus. From this plot it seems that, for fixed
$G$, the required time increases exponentially with $N$. However, we
believe that in the long run the behaviour will be worse than
exponential. This is suggested also by the fact that the ratio of non-isomorphic stable graphs over those created by our generation algorithm tends to zero as $G$ and $N$
grow (see~Figure~\ref{fig:plot_duplicated}).

\begin{figure}[t]
  \centering
  \includegraphics{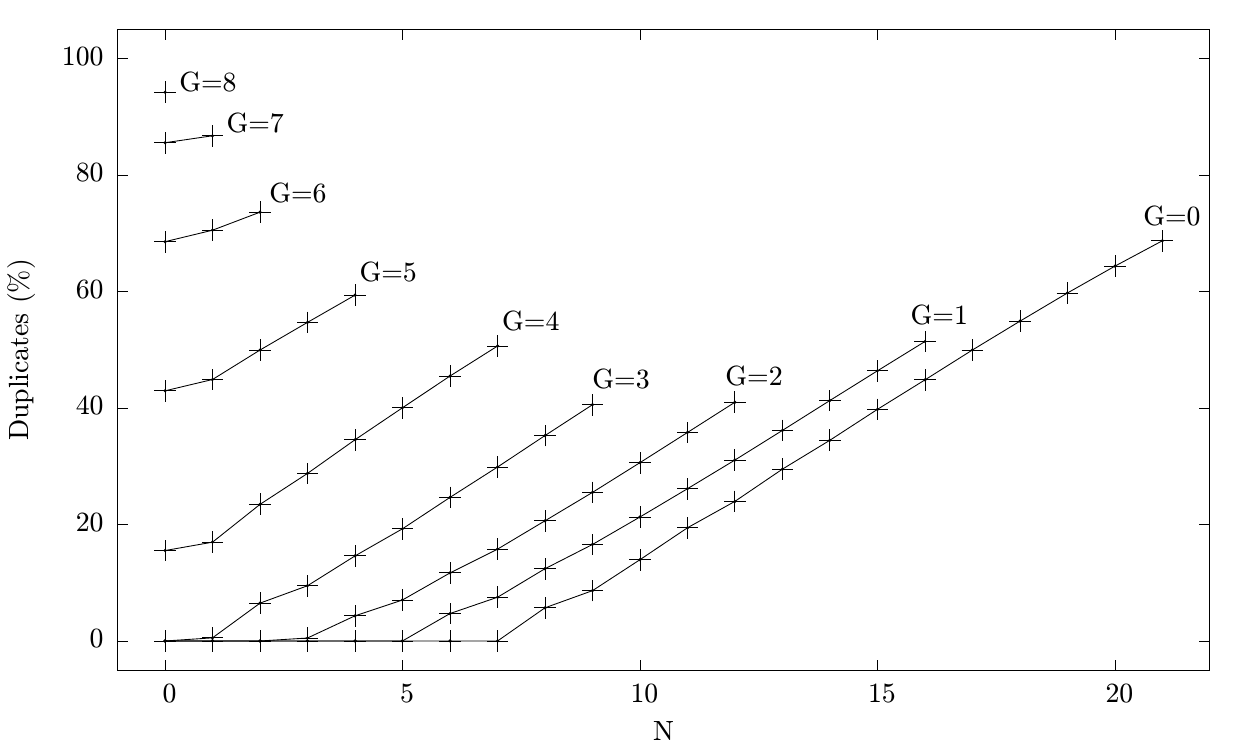}\vspace{-0.7\baselineskip}
  \caption{Ratio of duplicated graphs among the total number of those created by our generation algorithm.}\label{fig:plot_duplicated}
\end{figure}

More benchmarks and up-to-date computed results are available
at \texttt{boundary}'s webpage,
\href{http://people.sissa.it/~maggiolo/boundary/}
{\texttt{http://people.sissa.it/\~{}maggiolo/boundary/}}.

\section*{Acknowledgments}

Both the authors want to acknowledge their host institutions,
\textsc{sissa} and \textsc{kth}. The second author was partly
supported by the Wallenberg foundation. Both authors were partly
supported by \textsc{prin} ``Geometria delle variet\`a algebriche e
dei loro spazi di moduli'', by Istituto Nazionale di Alta
Matematica. The authors are also very grateful to Susha Parameswaran
for linguistic suggestions, and to the referees for suggesting further
improvements of the presentation.

\end{document}